\documentclass[colorlinks=true]{elsarticle}
  \usepackage{amssymb,amsfonts, amsthm, amstext,mathrsfs}
\usepackage[usenames,dvipsnames]{xcolor}

\usepackage[tbtags]{amsmath}

\usepackage[T1]{fontenc} 
\usepackage{hyperref} 
\usepackage{mathrsfs,times}
 \usepackage[text={130mm,200mm},centering]{geometry} 

\newcommand{\be}{\begin{equation}
\begin{aligned}[b]}
\newcommand{\ee}{\end{aligned}\end{equation}}
  \newcommand{\ben}{\begin{equation} \nonumber\begin{aligned}}

 \newcommand{\A}{\mathfrak{A}}
\newcommand{\B}{\mathfrak{B}}
 \newcommand{\R}{\mathbb{R}}
\newcommand{\C}{\mathfrak{C}}
\newcommand{\D}{\mathcal{D}}

\newcommand{\N}{\mathbb{N}}

\newcommand{\T}{\mathcal{T}}

\renewcommand{\d}{{\rm d}}

\newcommand{\dist}{{\rm dist}}

\newcommand{\bproof}{\begin{proof}}
\newcommand{\eproof}{\end{proof}}

\newcommand{\xto}{\xrightarrow}

\renewcommand{\(}{\left(}
\renewcommand{\)}{\right)}

\journal{a journal}

 \usepackage{thmtools}
\numberwithin{equation}{section}
\newtheorem{theorem}{Theorem}[section]
 \declaretheorem[name=Lemma, sibling=theorem]{lemma}
  
    \declaretheorem[name=Proposition, sibling=theorem]{proposition}
     \declaretheorem[name=Example, sibling=theorem]{example}

\theoremstyle{definition}
    \declaretheorem[name=Definition, sibling=theorem]{definition} 
     
\theoremstyle{remark}
     \declaretheorem[name=Remark, sibling=theorem]{remark}

    \usepackage{mathabx}
        
\newcommand{\ceqref}[1]{{\color{blue}\eqref{#1}}}
\newcommand{\cref}{\autoref}

\begin{document}
\begin{frontmatter}

 \title{Analysis of the dynamics of Caputo fractional differential equations}
  
\author[hust]{Hongyong Cui}\ead{h.cui@outlook.com} 
\address[hust]{School of Mathematics and Statistics \&  Hubei Key Laboratory of Engineering Modeling and Scientific Computing, \\ 
 Huazhong University of Science \& Technology, 
 Wuhan 430074, China}

\author[de]{Peter E. Kloeden} 
\ead{kloeden@math.uni-frankfurt.de}

\address[de]{Mathematisches Institut, Universität Tübingen, 72076 Tübingen, Germany}
 
 \author[jn]{Jie Xin}
 \ead{fdxinjie@sina.com}
 \address[jn]{School of Information Engineering, 
 Shandong Youth University of Political Science, Jinan 250103, China}

\begin{keyword} Global attractor, Caputo fractional differential equation, asymptotic behavior 
\MSC[2020] 34A08, 34D45, 45D05, 45M05.
\end{keyword}

\begin{abstract} 
   It is  known that a finite-dimensional Caputo fractional differential equation, though itself need not generate a semiflow,  can be represented as   a    Volterra integral equation which  generates an infinite-dimensional semiflow on the space $\C=C([0,\infty); \R^d)$ under the  standard compact-open topology. In this paper we  construct a  compact absorbing set  and    an  attractor for this semiflow on $\C$, and then  prove that  the attractor consists of equi globally H\"older continuous functions.  This strengthens the previous work of Doan \& Kloeden  \cite{DK21} where a bounded (with respect to a weighted norm) attractor was constructed.   
   \end{abstract}
%
%
\end{frontmatter}
     
  \tableofcontents

\section{Introduction}

We consider a Caputo fractional differential equation (FDE)  on $\R^d$ $(d\in \N)$: 
\begin{align}
\label{cfde}
\begin{cases} 
 ^C\! D_{0+}^\alpha x(t)   =g(x(t)),\quad t>0,  \\
x(0)  =x_0\in \R^d, 
\end{cases} 
\end{align}
where  $^C\! D_{0+}^\alpha$ denotes the Caputo type fractional differential operator in time with  index $\alpha\in (0,1)$ and $ g:\R^d \to \R^d  $  is  assumed  to be locally Lipschitz and to satisfy  the   dissipative condition
\begin{equation} \label{assu} 
 \big<x, g(x) \big>  \leq a - b\| x\|^2,\quad \forall x\in \R^d ,
\end{equation}  
for some constants $a,b>0$. Here  $\big< \cdot, \cdot\big> $  and  $\|\cdot\|$ are the   standard   inner product and norm of $\R^d$, respectively.

 FDEs have   enormous numbers of very interesting and novel applications in physics, chemistry, engineering, finance and other sciences in  modeling, e.g., the diffusion processes, mechanical properties of materials, signal processing and the behaviour of viscoelastic and viscoplastic materials under external influences, etc., see for instance \cite{diethelm}   where a number of references for these applications are listed. Many fractional generalizations of integral derivatives are available in the literature, while  two of them, the Riemann–Liouville derivative and the
Caputo derivative have been discussed most.  In view of dynamical systems, Caputo derivatives have an advantage  over the Riemann–Liouville that  Caputo FDEs use integral order initial conditions (see, e.g., \ceqref{cfde}) and Caputo derivatives of constant functions are zero. This makes Caputo FDEs  more physically meaningful in modeling the memory effects arising in  the real-world, see, e.g. monographs \cite{diethelm,kilbas,kubica}.  

However, unlike ODEs, the solutions of a Caputo FDE   do not satisfy the semiflow property in general:  let $x(t, x_0)$ be the solution of \ceqref{cfde} at time $t$, then,  generally,  
\[
  x(t+s, x_0) \neq x(t, x(s,x_0)), \quad t>s>0. 
\]
As a consequence,  Caputo FDEs need   not generate a dynamical system, and   the  mathematical analysis  on Caputo FDEs in the literature mainly emphasizes the well-possedness, regularity and numerical simulations of solutions (e.g., \cite{liu,WXK,XZC,zhou}), and, for the dynamical perspective, the stability and the stable manifolds which can  be studied by direct analysis on the solutions (e.g., \cite{congTT,LZ,TT,WFZ}). In other words, due to the lack of semiflow property,   the general theory of dynamical systems is not applicable to Caputo FDEs which causes a challenging  difficulty in  studying the dynamical  complexity of Caputo FDEs. 

Remarkable progress was made by Cong \& Tuan \cite{cong} who proved that  distinct solutions of some particular Caputo FDEs (e.g., scalar  or triangular systems of Caputo FDEs) would never intersect and hence the solution mapping $x_0\mapsto x(t,x_0)$ is invertible.  In this case the solutions  generate  a two-parameter semiflow $\varphi$, though the equation is autonomous, by 
\[
\varphi(t, \tau, x_0) := x(t, x^{-1} (\tau, x_0)) ,\quad \forall t\geq \tau \geq 0,\, x_0\in \R^d. 
\]
This makes it possible to study these Caputo FDEs via the theory of nonautonomous dynamical systems. However, such generation of a dynamical system is quite restrictive in general applications, especially in high dimensions. 

Most recently, Doan \& Kloeden \cite{DK21,doanklo2,DK24}  proposed a second possibility to associate a dynamical system to Caputo FDEs,  feeding the solution of the  finite-dimensional FDE into an infinite-dimensional semiflow,   see also the recent monograph \cite{DKT}. Essentially, their approach is based on the  observation that  the solution $x(t,x_0)$ of Caputo FDE \ceqref{cfde}  coincides with the solution 
 $x_f(t)$ of the 
 Volterra integral equation   \cite{MS}   in $\C: =C(\R^+, \R^d)   $ 
 \begin{equation}  \label{eq_Vb}
x_f(t)= f(t) +\frac{1}{\Gamma(\alpha)}
\int_0^t (t-s)^{\alpha -1} g(x_f(s)) \; \d s  
\end{equation} 
(here  $\Gamma (\cdot) $ is the Gamma function) when $ f(t) \equiv x_0$, that is, denoted by $id_{x_0}$ the constant function $id_{x_0}(t)\equiv x_0$, then 
\[
 x(t,x_0) =x_{id_{x_0} } (t). 
\]
 
Following    Sell \cite{sell},  Doan \& Kloeden \cite{DK21} showed that the Volterra operators
 $\{ \T_t\}_{t\geq 0}$ on $\C $ defined by
\begin{equation}  \notag
 (\T_t f) ( { \theta}) =f(t+{ \theta}) +\frac{1}{\Gamma(\alpha)}
\int_0^t (t+{ \theta} -s)^{\alpha -1} g(x_f(s)) \; \d s ,\quad f\in \C,\, \theta \in \R^+,
\end{equation}  
 where $x_f$ is the solution of Volterra integral equation \ceqref{eq_Vb}, form a    semiflow on $\C$. 
This means that the solutions $x(t,x_0)$ of the Caputo FDE \ceqref{cfde}  can be `fed' into the semiflow $\{ \T_t\}_{t\geq 0}$ on $\C $ acting on  constant functions $id_{x_0}$: 
\begin{equation}  \notag
 (\T_t  id_{x_0} )(0) =x(t,x_0) ,\quad \forall t\geq 0, \, x_0\in \R^d.  
\end{equation} 
Specifically, such solutions drive  the semiflow $\{ \T_t\}_{t\geq 0}$ . Hence,  the dynamical analysis of $\{ \T_t\}_{t\geq 0}$ on  constant functions  helps  understand that of the Caputo FDE. 

Let 
\begin{equation}  \notag
 \D :=\left \{  \mathfrak D=\bigcup_{x_0 \in D_0} id_{x_0}:  \text{ $D_0$ is a bounded subset of $\R^d$}   \right \} . 
\end{equation} 
Doan \& Kloeden \cite{doanklo2,DK24} then studied the $\D$-attractor   of the semiflow  $\{ \T_t\}_{t\geq 0}$.  For  technical reasons they equipped the space $\C$ with a weighted norm 
\[
 \|f\|_{\C_\alpha} := \|f(0)\| +\sum_{k=1}^\infty \frac {1} {2^k { k^\alpha} } \sup_{t\in [1/k, k]} \|f(t)\|  ,\quad \forall f\in \C .
\]
Under this norm they proved the existence of a {\em bounded} $\D$-absorbing set and  constructed a  {\em bounded} $\D$-attractor $\A$ for the semiflow  $\{ \T_t\}_{t\geq 0}$ in $\C$. Cui \& Kloeden \cite{CK} applied this technique  also in a nonautonomous setting. However, a counterexample (see \cref{ex})  shows that the space $\C$  under this norm is in fact {\em incomplete}. This incompleteness makes the standard attractor theory  inapplicable and  further analysis becomes technically difficult. 

In this paper, we study the Caputo FDE \ceqref{cfde}  by    the approach of Doan \& Kloeden but with certain  improvements.  We leave the norm $\|\cdot \|_{\C_\alpha}$ but use a carefully chosen metric  $d_\C$ of $\C$ which induces   the standard compact-open topology.    
 Under the same dissipativity condition \ceqref{assu} as in  \cite{doanklo2,DK24}  we prove that the semiflow  $\{ \T_t\}_{t\geq 0}$ has  in fact a {\em compact} $\D$-absorbing set (\cref{thm_main2}) and, as a consequence,  a {\em precompact}  $\D$-attractor $\A$ (\cref{thm_main3}) in $\C$.  In addition, the functions   in $\A$ are equi  globally   H\"older continuous with exponent $\alpha$ (\cref{thm_r}). 
All these results are based on a technical observation that the semiflow  $\{ \T_t\}_{t\geq 0}$ is asymptotically equi-continuous on every bounded set of  constant functions (\cref{prop_main2}).     \medskip
 
{\noindent \em Notation} 
Write $\R^+:=[0,\infty)$  and  denote by $\C : =C(\R^+, \R^d)   $ the space of continuous functions $f:\R^+\to \R^d,$ $ \theta \mapsto f(\theta) $. In particular, for any $x_0\in \R^d$ we denote by $id_{x_0}$ the constant function $id_{x_0} (\theta)\equiv x_0$ in $\C$. For any subset  $\mathfrak D$ of $
 \C$ and $\theta\geq 0$ we write the $\theta$-section  of $\mathfrak D$ as 
  \[
   \mathfrak D (\theta) :=\big\{ f(\theta): f\in \mathfrak D \big\} .
  \]
For any bounded  sets $A,B$ in a metric space $(
X, d_X)$, by $\dist_X(A,B)$ we denote the Hausdorff semimetric between them, i.e.,
$
 \dist_X(A,B) := \sup_{a\in A} \inf_{b\in B} d_{X} (a,b). 
$

\section{Caputo FDE in $\R^d$} 
 
 Due to the fact that  the solutions of Caputo FDE \ceqref{cfde}   do not satisfy a standard semiflow property,  
  they do not generate a standard dynamical system.  This makes it difficult to study  the dynamical behavior  of  the equation  in $\R^d$.

Nevertheless, one can study the solutions directly and can obtain some useful insights about the asymptotic behavior of the solutions.  For instance,  
Tuan \& Trinh \cite{TT} proved a useful inequality 
\begin{equation} \label{eq_TT} 
 ^C\! D_{0+}^\alpha \| x(t)\|^2 \leq 2\big< x(t), ^C\! D_{0+}^\alpha x(t) \big> ,
\end{equation}
for the sooutions of the Caputo FDE  \ceqref{cfde}.   Kloeden \cite{Kloeden} then  used this  inequality along with the dissipativity condition \ceqref{assu}  to derive that  
\begin{align*}
  ^C\! D_{0+}^\alpha \|x(t,x_0)\|^2  & \leq  
 2\big< x(t), g(x(t)) \big>  
 \leq 2a-2b\|x(t)\|^2, 
\end{align*} 
which leads to the energy inequality 
\begin{align} \label{eq_K}
 \|x(t,x_0)\|^2& \leq \|x_0\|^2 E_\alpha(-2b t^\alpha) +\frac ab \big(1- E_\alpha(-2b t^\alpha) \big)  ,
\end{align} 
where $E_\alpha (\cdot) $ denotes the one-parameter Mittag-Leffler function 
\[
E_{\alpha}(z):= \sum_{k=0}^{\infty}\frac{z^k}{\Gamma(\alpha k+1)} ,\quad \forall z\in \mathbb C .
\]

Since $E_\alpha(-2bt^\alpha)\to 0$ as $t\to \infty$ (see, e.g., \cite{kilbas}), 
  from inequality \ceqref{eq_K}  it follows \begin{lemma} \label{lem_abB}
For any $R>0$ there exists a $T^*=T^*(R)>0$ such that the solutions $x(t,x_0)$  of  the Caputo FDE \ceqref{cfde} with $\|x_0\| \leq R$ satisfy
\[
 \|x(t, x_0)\|^2 \leq \frac ab+1,\quad \forall t\geq T^*(R). 
\]
\end{lemma} 

This means   the Caputo FDE \ceqref{cfde}  has a compact `absorbing set' 
in $\R^d$  given by
\begin{equation} \label{eq5.25.1} 
 B^*:= \left \{ y\in \R^d:  \|y\|\leq \sqrt{ \frac a b+ 1} =:\delta^* \right \} .
\end{equation}
This however does not lead to a global attractor in $\R^d$ because the Caputo FDE does not generate a dynamical system in $\R^d$.  But it does have a precompact attracting set $\Omega^*$ included in $B^*$, given by
\begin{equation} 
 \Omega^* :=  \Big\{ y\in \R^d: \exists \{x_{n}\}_{n\in \N}  \text{ bounded and }  t_n\to \infty \text{ such that }  x(t_n,x_{n})\to y \Big \}    . \label{def_O}
\end{equation}

 Due to the lack of semiflow property of solutions,   no more information about this attracting set is until   now available. For instance,  the compactness, invariance,   or if  solely determined by the solutions starting from the absorbing set $B^*$ remain open.

\section{Feeding the Caputo FDE solution into  the Volterra semiflow} 
 
  Since  the Caputo FDE itself does not generate a semiflow,  motivated by Doan \& Kloeden \cite{DK24} we now study it  by feeding the solutions into an infinite-dimensional semiflow, the Volterra semiflow  on $\C=C(\R^+, \R^d) $, as a bridge.

\subsection{Volterra semiflow on $\C $} 

A key observation is that the   solution of Caputo FDE  \ceqref{cfde}  is an integral equation (with memory) 
 \begin{equation}\label{AIE}
x(t)=x_0+\frac{1}{\Gamma(\alpha)}
\int_0^t (t-s)^{\alpha -1} g(x(s)) \; \d s ,\quad t\geq 0
\end{equation} 
which is a particular case of the Volterra integral equation   of $f$ \begin{equation} \label{eq_V}
x_f(t)= f(t) +\frac{1}{\Gamma(\alpha)}
\int_0^t (t-s)^{\alpha -1} g(x_f(s)) \; \d s  
\end{equation} 
 in $\C =C(\R^+, \R^d)   $ 
when $ f(\theta) \equiv x_0$.  

 Doan \& Kloeden \cite{DK21} using Sell \cite{sell}  showed that the Volterra operators $\{ \T_t\}_{t\geq 0}$ on $\C $ defined by
\begin{equation} \label{dy_sell}
 (\T_t f) ( { \theta}) =f(t+{ \theta}) +\frac{1}{\Gamma(\alpha)}
\int_0^t (t+{ \theta} -s)^{\alpha -1} g(x_f(s)) \; \d s ,\quad f\in \C,\, \theta \in \R^+,
\end{equation}  
form a    semiflow on $\C$, where $x_f$ is the solution of Volterra integral equation \ceqref{eq_V}.  In addition, each $\T_t$ is a  continuous operator in $\C$ with respect to the compact-open topology of $\C$ (i.e.,  the uniform convergence topology on each compact interval)  induced by the metric
\begin{equation} \label{def_rho}
 \varrho(f,h) := \sum_{k=1}^\infty \frac 1{2^k}   \frac{ \sup_{\theta \in[0,k]} \|f(\theta)-h(\theta)\|}{1+\sup_{\theta\in[0,k]} \|f(\theta)-h(\theta)\|},\quad \forall f,h\in \C.
\end{equation} 
Call this semiflow $\{ \T_t\}_{t\geq 0}$  a {\em  Volterra  semiflow}, and    write it as $\T$ for simplicity.

For  a constant function $id_{x_0}(\theta)\equiv x_0$ in $\C $,  the state $\T_tid_{x_0}$ at $t$ of semiflow $\T$  starting from $id_{x_0}$  is then given by
\begin{equation} \label{eq5.25.5}
(\T_t id_{x_0} ) ( { \theta}) =x_0 + \frac{1}{\Gamma(\alpha)}
\int_0^t (t+{ \theta} -s)^{\alpha -1} g(x(s)) \; \d s ,\quad   \theta\geq 0. 
\end{equation}
 Moreover,  the $0$-section  coincides the solution $x(t,x_0)$ of Caputo FDE \ceqref{cfde}, i.e., 
\begin{equation} \label{eq5.25.2}
 (\T_t  id_{x_0} )(0) =x(t,x_0) ,\quad \forall t\geq 0, \, x_0\in \R^d ,
\end{equation} 
by which  the solutions of Caputo FDE \ceqref{cfde} are now fed into the  semiflow  $  \T $ on $\C$, see \cref{fig1}.  
 Therefore, in order to study the   dynamical behavior of the original problem, Caputo FDE \ceqref{cfde}, inspired by \ceqref{eq5.25.2}  one can study the Volterra semiflow $ \T $  but with attention restricted on constant functions $id_{x_0}$.  

  \begin{figure}
	\centering{
  \includegraphics[width=.9\textwidth]{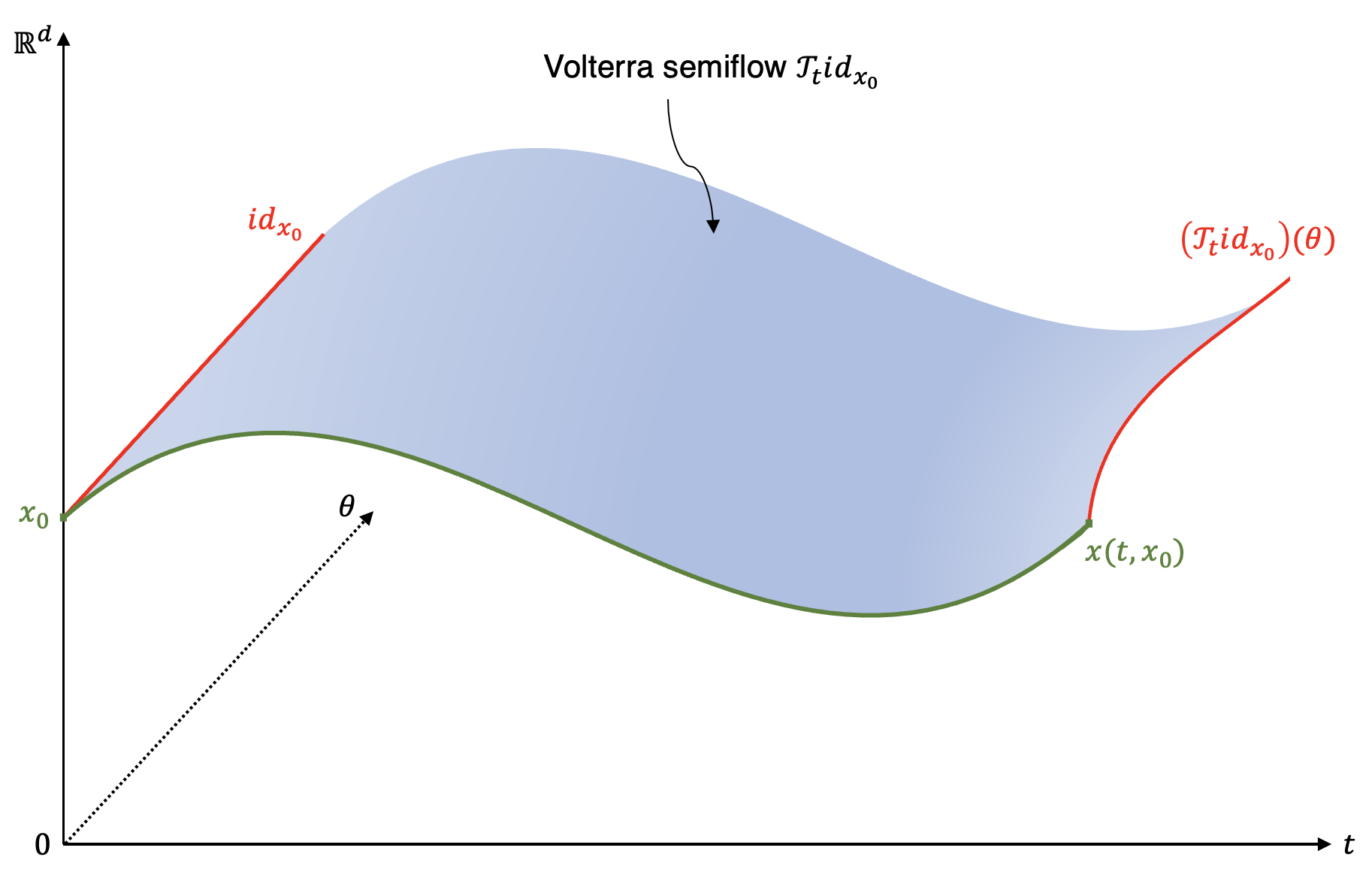}
	\caption{  
	 Solution $x(t,x_0)$ of Caputo FDE  fed into a  Volterra  semiflow $\T$ on $\C=C(\R^+, \R^d)$} \label{fig1}
	}
	\end{figure}

\subsection{An admissible  metric of $\C$} 

Doan \& Kloeden \cite[Lemma 3]{DK24} proved under a  global Lipschitz condition of $g$ the following result.

\begin{lemma}\label{lem:bdd} 
 For any $R>0$ there exists a  time $T=T(R)>0$ such that 
\[
\|(\T_t id_{x_{0}} )(\theta)  \| \leq  c  \theta^\alpha + \delta^*,\quad \forall t\geq T ,\, \| x_0\| \leq R,
\]
for some   $ c>0$ independent of $\R$, 
where $\delta^*$ is the radius of $B^*$ given by \ceqref{eq5.25.1}. 
\end{lemma} 
  
This indicates that  the state   $ (\T_t id_{x_{0}}) (\theta) $  at each $t$ may have  unbounded  `tail' as $\theta\to \infty$. 
Doan \& Kloeden \cite{DK24} hence  considered a subspace $\C_\alpha $ of $\C $: 
$$ 
 \C_\alpha:=\big\{ f\in \C: \|f\|_{\C_\alpha} <\infty \big \}, 
$$
endowed with the  weighted  norm
\[
 \|f\|_{\C_\alpha} := \|f(0)\| +\sum_{k=1}^\infty \frac {1} {2^k { k^\alpha} } \sup_{\theta \in [1/k, k]} \|f(\theta)\| ,
\]
and  studied the attracting set of $ \T$  in $\C_\alpha$. Unfortunately,  the space $\C_\alpha$ is not complete under the norm $d_{\C_\alpha}$, as illustrated by the following example.

\begin{example} \label{ex} 
Define a sequence $\{f_n\}_{n\in \N} $ of  functions in $C([0,\infty), \R)$ by 
\begin{equation} \notag
 f_n(\theta) :=\begin{cases}
 1-n\theta, &  0\leq \theta \leq  1/n , \\
 0, & \theta  > 1/n ,
 \end{cases}  
 \quad \text{and} \quad 
 \delta_0(\theta) :=\begin{cases}
1, & \theta=0,\\
0, & \theta>0 .
\end{cases}
\end{equation} 
Then, since  $\|f_n(\theta)-\delta_0(\theta)\| \leq 1$ for all $n\in \N$ and $\theta\geq 0$,  from the Weierstrass dominated convergence theorem it follows
\begin{align*}
\lim_{n\to \infty}  \! \(
\sum_{k=1}^\infty \frac {1} {2^k { k^\alpha} } \sup_{t\in [1/k, k]}  \! \|f_n(t)-\delta_0(t)\| \)
&= \sum_{k=1}^\infty \frac {1} {2^k { k^\alpha} } \(\lim_{n\to \infty}  \sup_{t\in [1/k, k]}  \! \|f_n(t)-\delta_0(t)\|  \) 
 \\[0.8ex]
& =0.
\end{align*}  
Therefore,  as $f_n(0)\equiv 1 =\delta_0(0)$, 
\[
 \| f_n - \delta_0\|_{\C_\alpha} =  \|f_n(0) -\delta_0(0)\| + 
\sum_{k=1}^\infty \frac {1} {2^k { k^\alpha} } \sup_{t\in [1/k, k]} \|f_n(t)-\delta_0(t)\|  \to 0
\]
 as $n\to \infty$. But   $\delta_0 \notin \C_\alpha$, so $\C_\alpha$ is not complete.
\end{example}

As a consequence, the standard global attractor or dynamical system theory does not apply directly to $\T $ on  $(\C_\alpha, d_{\C_\alpha})$.

\medskip 

We now leave  $\C_\alpha$ and come back to  the original state space $\C=C([0,\infty), \R^d)$ of $\T$. In order to study the dynamics of Caputo FDE \ceqref{cfde} via $\T$ we need choose an appropriate metric on $\C$. 
We take
\begin{equation} 
 d_{\C} (f,h) : =\|f(0)-h(0)\|+   \sum_{k=1}^\infty \frac 1{2^k}   \frac{ \sup_{\theta \in[0,k]} \|f(\theta)-h(\theta)\|}{1+\sup_{\theta\in[0,k]} \|f(\theta)-h(\theta)\|},\quad \forall f,h\in \C . \label{def_dc}
\end{equation} 
Then  the metric $ d_\C$ induces the same   topology of $\C$, i.e., the compact-open topology, as  $\varrho$. More precisely,
  \begin{lemma} \label{lem_top}
A sequence  $\{f_n\}_{n\in \N}$  in $(\C, d_{\C})$ converges to $f$    if and only if    for any  compact interval $ [0, k]\subset [0,\infty)$, $ \sup_{\theta \in [0,k] } |f_n (\theta)- f(\theta) | \to 0$ as $n\to \infty$.
\end{lemma} 
\bproof 
It suffices to notice that $d_\C(f_n, f)\to0$ if and only if $\varrho(f_n,f)\to 0$. 
\eproof

\begin{remark} \label{rmk1} 
The space $\C$ is now unbounded with respect to $d_\C$, and a set $\mathfrak D $ in $ \C$ is bounded in $d_\C$ if and only if the 0-section $\mathfrak D(0):=\{f(0): f \in \mathfrak D\}$ of $\mathfrak D$ is bounded in $\R^d$. This is important for sets of identical functions  $id_{x_0}$ with  $x_0$ in bounded subsets of $\R^d$.  Since the  Volterra semiflow  $\T$ on $\C$ carries the information of the Caputo FDE \ceqref{cfde} exactly by the 0-sections of its trajectories (see \ceqref{eq5.25.2}), this metric $d_\C$ is more admissible   to study the asymptotic behavior of \ceqref{cfde} than the standard metric $\varrho$. 
\end{remark}

In the sequel, when we talk about the metric of $\C$ we mean $d_\C$ (rather than $\varrho$).

  \section{Dynamical behavior of the Volterra semiflow} 
We now study the attractor theory for the Volterra semiflow on $(\C, d_\C)$, as a bridge to study the dynamics of Caputo FDE \ceqref{cfde}. Inspired  by \ceqref{eq5.25.2}, we only need restrict our attention   on constant functions $id_{x_0}$.  
Let
\begin{equation} \label{def_D}
 \D :=\left \{  \mathfrak D=\bigcup_{x_0 \in D_0} id_{x_0}:  \text{ $D_0$ is a bounded subset of $\R^d$}   \right \} . 
\end{equation} 
That is,  every  element $\mathfrak D $ of    universe  $\mathcal D$ is a bounded set $\mathfrak D$ in $\C$   `stretched' from a bounded set $D_0$ in $\R^d$. 
We then study the $\D$-attractor of Volterra semiflow $\T$, which is expected to attract every $\mathfrak D$ in $\D$ under  the Volterra semiflow  $\T$.  The collection  $\mathcal D$ is then  called the {\em attraction universe} of the attractor. 

We make some technical  preparations first. They will be   crucial for constructing a compact absorbing set of $\T$  later. 

  \subsection{Asymptotic equi-continuity}

 Since $g$ is locally Lipschitz and  the Caputo FDE \ceqref{cfde}  has a compact `absorbing set'  $B^*$  (given by \ceqref{eq5.25.1}), for every $R > 0$  we define the constant 
 \begin{equation} \label{def_G}
 G_R:=  \sup_{t\geq 0, \, \| x_0\|\leq R }    \sqrt{ d\, } \, \big \| g(x(t, x_0)) \big \| ,
 \end{equation} 
 where  $ d  $ is the dimension of $\R^d$.  Particularly for $x\in B^*$ (i.e., $\|x\|\leq \delta^*$,  see \ceqref{eq5.25.1}),  we define the constant 
 \begin{equation} \label{Gs}
  G_* := \sup_{x\in B^* }   \sqrt d \, \big \| g(x) \big \| .
 \end{equation} 

   The following result strengthens  Doan \& Kloeden \cite{DK24} where an $R$-dependent upper bound  was derived, while our bound here is  independent of $R$. 
   This $R$-independence lies in the center of our analysis later.

 \begin{proposition} \label{prop_main}
  For  any $R>0$ there exits a $T=T(R,\alpha)>0$   such that 
 \begin{align}  \label{eq6.13.2} 
 \bigg \| \frac{\d }{\d \theta} (\T_t id_{x_0} )(\theta) \bigg\|   \leq      \frac{ G_* }{\Gamma(\alpha)}   \(  \theta^{\alpha-1} + \theta^{\frac{\alpha-1} 2}  \)   ,\quad \forall \theta >0, 
 \end{align} 
 uniformly for all $t\geq T$ and $\|x_0\|\leq R$, where $G_*$ is the constant defined  by \ceqref{Gs}. 
 \end{proposition} 
  
 \bproof
 Take arbitrarily $\theta>0$. 
 For any $t>0$, 
 \begin{align}
 \bigg \| \frac{\d }{\d \theta} (\T_t id_{x_0} )(\theta)  \bigg\| 
 & = \frac{1}{\Gamma(\alpha)} \, \bigg \| \frac{\d }{\d \theta} \int_0^t (t+\theta-s)^{\alpha-1} g(x(s, x_0)) \, \d s \, \bigg \|  \quad \text{(by \ceqref{eq5.25.5})} \notag \\
 & = \frac{1-\alpha}{\Gamma(\alpha)}  \,  \bigg \|  \int_0^t (t+\theta-s)^{\alpha-2} g(x(s, x_0)) \, \d s \,  \bigg\| \notag  \\ 
  & \leq   \frac{1-\alpha}{\Gamma(\alpha)}  \, \bigg \|  \int_{t/2} ^t (t+\theta-s)^{\alpha-2} g(x(s, x_0)) \, \d s \,  \bigg\| \notag \\
 &\quad +  \frac{1-\alpha}{\Gamma(\alpha)} \, \bigg\|  \int_0^{t/2}  (t+\theta-s)^{\alpha-2} g(x(s, x_0)) \, \d s \,  \bigg\|   . \label{eq6.5.1}
 \end{align}

 Now, let $t\geq 2T^*(R)$ and  $ \|x_0 \| \leq R$,  then \cref{lem_abB} indicates that  $$\|x(s,x_0) \| \leq \delta^*,\quad \forall s \geq t/2, $$ so
 \begin{align}
  \left \|  \int_{t/2}^t (t+\theta-s)^{\alpha-2} g(x(s, x_0)) \, \d s  \right\| 
 & \leq G_*  \int_{t/2}^t (t+\theta-s)^{\alpha-2}  \ \d s  \quad \text{(by \ceqref{Gs})}  \notag  \\
 &=  \frac{G_*}{1-\alpha} \left[  { \theta^{\alpha-1}- \(\theta + \frac t2\)^{\alpha-1} } \right ]\notag 
 \\ 
 &<\frac{G_*}{1-\alpha} \theta^{\alpha -1}. \label{eq5.27.3} 
 \end{align}
 On the other hand, 
 \begin{align}
  \left \|  \int_0^{t/2}    (t+\theta-s)^{\alpha-2} g(x(s, x_0)) \, \d s  \right\| 
 & \leq G_R \int_0^{t/2}  (t+\theta-s)^{\alpha-2}  \ \d s  \quad \text{(by \ceqref{def_G})}   \notag
 \\ 
 &=  \frac {G_R}  {1-\alpha}    \left [ \( \frac t2 +\theta\)^{\alpha-1} -(t+\theta)^{\alpha-1} \right]  \notag \\ 
  & < \frac {G_R}  {1-\alpha}    \( \frac t2 +\theta\)^{ \frac{\alpha-1}2 } \( \frac t2 +\theta\)^{ \frac{\alpha-1}  2}   \notag\\
 &< \frac {G_R}  {1-\alpha}     \( \frac t2\)^{ \frac{\alpha-1}2 } \theta^{ \frac{\alpha-1}  2}      .  \label{eq5.27.1}
 \end{align}
 Therefore, by \ceqref{eq6.5.1}--\ceqref{eq5.27.1} it yields
\begin{equation}  \notag
 \left \| \frac{\d }{\d \theta} (\T_t id_{x_0} )(\theta) \right\|  
  \leq  \frac{G_*}{\Gamma(\alpha)} \theta^{\alpha -1} 
  + \frac {G_R}  {\Gamma(\alpha)}     \( \frac t2\)^{ \frac{\alpha-1}2 } \theta^{ \frac{\alpha-1}  2}  
 \end{equation} 
 for all $ t\geq  2T^*(R) ,$ $ \theta >0$, and $\|x_0\|\leq R$.

 Since  $\alpha-1 <0$, there exists  a   $T=T(R,\alpha)  \geq 2T^*(R) $ such  that 
 $$    {G_R}    \( \frac t2 \)^{\frac{\alpha-1}2 } \leq   {G_*}    ,\quad \forall t\geq T,$$   
 and then  \ceqref{eq6.13.2}  follows. 
 \eproof
 
 \cref{prop_main} indicates that
 \[
  \left \| \frac{\d }{\d \theta} (\T_t id_{x_0} )(\theta) \right\|        \to 0 \quad \text{as $ \theta\to \infty$} ,
  \]
  uniformly for $t\geq T_R$ and $\|x_0\|\leq R$. Intuitively, this means the `tail' of $(\T_t id_{x_0} )(\theta)$ will finally (for $t$ large) stay calm  as $\theta \to \infty$,  
 but it does not give a  bound of this $\theta$-derivative as $\theta\to 0^+$, so it remains unclear whether $(\T_t id_{x_0} )(\theta)$  can vary dramatically near the 
 `head' (as $\theta\to 0^+ $).  We now study this problem.   
 
   The following inequalities are  elementary.

\begin{lemma} \label{ineq}
For any $p > q > 0$ and $\gamma \in (0,1)$,
\[
p^{\gamma} - q^{\gamma}  < (p- q)^{\gamma}.
\]
\end{lemma}

    \begin{lemma} \label{ineq2}
       For any $p,q,\gamma>0$, 
\[
 p^\gamma -q^\gamma
 = \gamma (p-q) \int^1_0\big(q+\tau (p-q) \big)^{\gamma-1}\ \d \tau .
\]
   \end{lemma} 
   \begin{proof}
  It suffices to notice that 
\[
   \int^1_0\big(q+\tau (p-q) \big)^{\gamma-1}\ \d \tau 
    =\frac 1{p-q} \int_q^p s^{\gamma -1} \ \d s =\frac 1{ \gamma(p-q) }  \( p^\gamma -q^\gamma\).  \qedhere 
\]
   \end{proof}

 \begin{proposition} \label{prop_main2} 
For any $R>0$ there exists a $T=T(R, \alpha)>0$ such that,    for all $ \theta_1, \theta_2\geq 0 $  with $|\theta_1  -\theta_2| \leq 1$,   
     \begin{align} \notag
  \big  \|  (\T_t id_{x_0} )(\theta_1)  - (\T_t id_{x_0} )(\theta_2)  \big  \|
 \leq \frac{2G_*}{\alpha \Gamma(\alpha)}    | \theta_1 - \theta_2| ^\alpha ,\quad \forall t\geq T,
 \end{align}   
 uniformly for all $\| x_0\|\leq R$, 
   where $G_*$   is an $R$-independent constant defined by \ceqref{Gs}.
 \end{proposition} 
 \begin{proof}
 Let $t>0$ and $  \theta_1 >\theta_2\geq 0 $ be arbitrarily given. 
 By \ceqref{eq5.25.5}, 
\begin{equation} \label{eq6.7.7}
 \big \|  (\T_t id_{x_0} )(\theta_1)  - (\T_t id_{x_0} )(\theta_2) \big \|
 =  \frac{1}{\Gamma(\alpha)} \left\| \, \int_0^t    \eta_{\theta_1,\theta_2}(s)  g(x(s,x_0)) \, \d s \, \right \|  
\end{equation} 
 with 
$
  \eta_{\theta_1,\theta_2}(s):=   (t+\theta_1 -s)^{\alpha -1} -(t+\theta_2 -s)^{\alpha -1}   . 
 $

 Recall from \cref{lem_abB} that for the given $R$   there is a $T^*=T^*(R )>0$ such that  $x(s,x_0) \in B^*$ for all $\| x_0\|\leq R$ and $s\geq T^*$. 
 Now we consider $t\geq 2 T^*  $, then 
  \ben
 \bigg \| \int_{t/2}^t   \eta_{\theta_1,\theta_2}(s)  g(x(s,x_0)) \, \d s\, \bigg \| 
&\leq G_*  \int_{t/2}^t | \eta_{\theta_1, \theta_2}(s) | \, \d s   \\ 
 &  = 
 G_*   \int_{ t/2}^t  \Big[
 (t+\theta_2 -s)^{\alpha -1} -(t+\theta_1 -s)^{\alpha -1} \Big] \,  \d s , \label{eq6.7.4} 
 \ee
 where $G_*$ is the constant defined by \ceqref{Gs}.  After  some elementary  calculation, 
 \[
 \int_{t/2}^t 
 (t+\theta -s)^{\alpha -1}  \,  \d s = \frac 1 \alpha \left[ \Big( \frac t2 +\theta \Big)^\alpha -\theta^\alpha \right ], \quad \forall \theta\geq 0, 
 \]
 and hence 
 \be
 &  G_*  \int_{t/2}^t  \Big[
 (t+\theta_2 -s)^{\alpha -1} -(t+\theta_1 -s)^{\alpha -1} \Big] \,  \d s \\
 &\quad =  \frac {G_*} \alpha \left[ \Big( \frac t2 +\theta_2 \Big)^\alpha -\theta_2^\alpha  +\theta_1^\alpha -\Big( \frac t2 +\theta_1 \Big)^\alpha \right ] \\
 & \quad\leq  \frac  {G_*} \alpha  \Big( \theta_1^\alpha - \theta_2^\alpha  \Big) \quad \text{(since $\theta_1 >\theta_2\geq 0$)} \\
 &\quad \leq \frac  {G_*} \alpha (\theta_1 -\theta_2)^\alpha  \quad \text{(by \cref{ineq})}  ,\quad t\geq 2T^*. \label{eq6.7.5} 
 \ee 

On the other hand,  still for $ t\geq 2T^*$,
 \ben
 \bigg \| \int_0^{t/2}     \eta_{\theta_1,\theta_2}(s)  g(x(s,x_0)) \, \d s \, \bigg \| 
 &   \leq
 G_R  \int_0^{t/2} \big |  \eta_{\theta_1,\theta_2}(s) \big | \, \d s  \\
  &  =  
 G_R  \int_0^{t/2}  \Big[
 (t+\theta_2 -s)^{\alpha -1} -(t+\theta_1 -s)^{\alpha -1} \Big] \,  \d s ,
 \ee
 where $G_R$ is the constant defined by \ceqref{def_G}. 
 Since 
 \[
  \int_0^{t/2} (t+\theta -s)^{\alpha -1} \, \d s = \frac 1 \alpha \left[ (t+\theta)^\alpha -\Big(\frac t2 +\theta \Big)^\alpha\right] , \quad\forall \theta\geq 0,
 \]
 it follows 
  \be
& \left\|\, \int_0^{t/2}     \eta_{\theta_1,\theta_2}(s)  g(x(s,x_0)) \, \d s \, \right \|  \\
&\quad  \leq 
 \frac{G_R}{\alpha} 
 \left[ ( t+\theta_2)^\alpha- \Big(\frac t2+\theta_2\Big)^\alpha -(t+\theta_1)^\alpha +  \Big(\frac t2+\theta_1\Big)^\alpha  \right]\\
 &\quad  \leq 
 \frac{G_R}{\alpha} 
 \left[      \Big(\frac t2+\theta_1\Big)^\alpha -\Big(\frac t2+\theta_2\Big)^\alpha   \right]  . \label{eq6.7.3}
 \ee
 By \cref{ineq2},
 \ben
   \( \frac t2+\theta_1 \)^\alpha - \(\frac t2 +\theta_2 \)^\alpha
   &= \alpha(\theta_1 -\theta_2) \int_0^1 \Big[ \Big( \frac t 2+\theta_2\Big) +\tau(\theta_1-\theta_2)\Big]^{\alpha-1} \, \d \tau \\
   &\leq  \alpha(\theta_1 -\theta_2) \( \frac t 2+\theta_2 \)^{\alpha-1} \quad \text{(taking $\tau\equiv 0$)} \\
   &\leq  \alpha \( \frac{  t}2 \)^{\alpha -1}  (\theta_1 -\theta_2) \quad \text{(since $\theta_2\geq 0$)} \\
   &\leq  2 \alpha  t^{\alpha -1}  (\theta_1 -\theta_2) .
 \ee 
   Therefore, returning to  \ceqref{eq6.7.3}, 
 \be  \label{eq6.7.10}
 \left\| \, \int_0^{t/2}     \eta_{\theta_1,\theta_2}(s)  g(x(s,x_0)) \, \d s \, \right \|  \leq {2 G_R}  t^{\alpha-1}   (\theta_1 -\theta_2) .
\ee

 Combining  \ceqref{eq6.7.7}, \ceqref{eq6.7.5} and \ceqref{eq6.7.10}  yields
 \be 
 \big \|  (\T_t id_{x_0} )(\theta_1)  - (\T_t id_{x_0} )(\theta_2) \big \|  
&=  \frac{1}{\Gamma(\alpha)} \left\| \, \int_0^t    \eta_{\theta_1,\theta_2}(s)  g(x(s,x_0)) \, \d s \, \right \|  
\\
&   \leq
 \frac {G_* }{\alpha \Gamma(\alpha )}  (\theta_1  - \theta_2 )^\alpha + \frac{2G_R t^{\alpha-1} }{\Gamma(\alpha)}(\theta_1 -\theta_2)  \label{eq6.7.11}
\ee
for all $t\geq  2T^*$.  
 This means,   there exists a $T=T(R,\alpha)\geq 2T^* $ such that 
 \[
  2G_R t^{\alpha-1}  \leq  \frac {G_* }{\alpha} ,\quad \forall t\geq T, 
 \]
 and then  
  \ben
 \big \|  (\T_t id_{x_0} )(\theta_1)  - (\T_t id_{x_0} )(\theta_2) \big \| \leq
 \frac {G_* }{\alpha \Gamma(\alpha )}    \Big[ (\theta_1  - \theta_2 )^\alpha  +(\theta_1 -\theta_2) \Big]
  ,\quad \forall t\geq  T.   \label{eq6.7.9}
\ee
If we let in addition that  $|\theta_1 -\theta_2| \leq 1$, then 
 \[
  \big \|  (\T_t id_{x_0} )(\theta_1)  - (\T_t id_{x_0} )(\theta_2) \big \|   \leq
 \frac {2G_* }{\alpha \Gamma(\alpha) }      (\theta_1  - \theta_2 )^\alpha   
  ,\quad \forall t\geq 2 T^* ,
\]
which completes the proof.
  \end{proof}

 \subsection{Constructing a compact $\D$-absorbing set} 
    
    Now we are ready to construct a compact $\D$-absorbing set for the Volterra semiflow $\T$. Continuous functions in $\C$ with property
    \begin{equation} 
      \|f(\theta_1) -f(\theta_2)\| \leq \frac {2G_* }{\alpha \Gamma(\alpha) }     |\theta_1  - \theta_2 |^\alpha \quad  \text{ for  } |\theta_1-\theta_2|\leq 1 , \label{ab_f}
      \end{equation} 
    where  $G_*$ is the constant defined by \ceqref{Gs}, 
 will play a central role. 
 
 Define
\[
 \B_0 :=\Big \{ f\in \C : \| f(0)\| \leq \delta^* \text{ and satisfies property \ceqref{ab_f}}
 \Big \} ,
\]
where $\delta^*= \sqrt{ a/b +1} $ is the radius of the ball $B^*$  in $\R^d$ defined by \ceqref{eq5.25.1}.

\begin{theorem} \label{thm_main2}     
 $\B_0$ is a $\D$-absorbing set  of the Volterra semiflow  $\T$ which is  precompact with respect to the metric $ d_\C $.  
 \end{theorem} 
 \begin{proof} 
\emph{a) $\D$-absorption.}  We first prove that $\B_0$  absorbs every  $ \mathfrak{D} = \cup_{x_0\in D_0} id_{x_0} $ in $ \mathcal D$.  Since $D_0$ is bounded in $\R^d$,   without loss of generality we assume $ \|D_0\|\leq R$ for some $R>0$.  By \cref{lem_abB}, $D_0$ is absorbed by $B^*$, i.e., there exists a $T_1=T_1(R)>0$ such that
 \[
  \bigcup_{x_0\in D_0}  x(t, x_0) \subset  B^*,\quad \forall t\geq T_1 .
 \]
 By \ceqref{eq5.25.2},  this means 
 \begin{align*}
 (\T_t  \mathfrak{D} )(0) = \bigcup_{x_0\in D_0} ( \T_t id_{x_0} )(0)
 = \bigcup_{x_0\in D_0} x(t,x_0) 
 \subset B^*,\quad \forall t\geq T_1 ,
  \end{align*} 
  or, equivalently, 
   \begin{align*}
 \| (\T_t  \mathfrak{D} )(0)  \|  \leq \delta^*,\quad \forall t\geq T_1 .
  \end{align*} 
  
  On the other hand,  \cref{prop_main2}   indicates that there is a $T=T(R)$ such that  
  \begin{align} \notag
  \big  \|  (\T_t id_{x_0} )(\theta_1)  - (\T_t id_{x_0} )(\theta_2)  \big  \|
 \leq \frac{2G_*}{\alpha \Gamma(\alpha)}    | \theta_1 - \theta_2| ^\alpha ,\quad \forall t\geq T,
 \end{align}   
 uniformly  
  for all $id_{x_0} \in  \mathfrak{D}$ and $|\theta_1-\theta_2|\leq 1$. 
 Hence, $\B_0$ absorbs $ \mathfrak{D}$  under $\T$.
 
 \noindent 
 \emph{b) Precompactness}. Now we prove the precompactness of $\B_0$ with respect to the metric $d_\C$.  
 Let $\{ f_n\}_{n\in \N}$ be a sequence in $\B_0$, we need to show that it has a Cauchy subsequence. Since $\sup_{n\in \N} \|f_n(0)\| \leq  \delta^*$ and $\R^d$ is finite-dimensional, we have a Cauchy  subsequence $ \{ f_n(0) \}_{n\in \N}$ (after relabeling)  in $\R^d$. 
 
  In the following it suffices to prove that   $\{ f_n \}_{n\in \N} $ is precompact  in the compact-open topology of $C([0,\infty), \R^d)$ by the Ascoli--Arzel\`a  theorem. 
 First, it is clear that for every $\theta\geq 0$,   $\{ f_n (\theta) \}_{n\in \N} $ is   bounded since
 $\{ f_n (0) \}_{n\in \N} $ is   bounded and  $ \theta\mapsto \| f_n(\theta) \|$ have uniformly bounded growth rate.  Second,  the property \ceqref{ab_f} indicates that for any compact set $[a,b]\subset [0, \infty)$ the sequence  $\{ f_n \}_{n\in \N} $ is  equi-continuous on $[a,b]$.  Therefore, the Ascoli--Arzel\`a   theorem on locally compact topological spaces  ensures that $\{ f_n \}_{n\in \N} $ is precompact in the compact-open topology of $C([0,\infty), \R^d)$ and, equivalently, in the metric $d_\C$. 
 \end{proof} 
 
  Take the $ d_\C$-closure of $\B_0$, i.e.,  let
 \[
  \B := cl_{\C} (\B_0). 
 \]
 Then $\B$ is a compact absorbing set  of $\T$, which   need not belong to the universe $\D$.

  \subsection{The $\D$-attractor} 
  
  Now we study the $\D$-attractor of the Volterrra semiflow $\T$ on $\C=C\big( [0,\infty); \R^d \big) $. 
Let 
 \begin{align}
 \A :=    \left\{ f\in \C
 \left| 
 \begin{aligned}  &\ \exists \text{ a bounded sequence } \{x_{n}\}_{n\in \N} \subset \R^d \text{  and } t_n\to \infty  \\
 & \text{ such that }  d_\C(\T_{t_n} id_{x_{n}} , f )\to0 
 \text{ as } n\to \infty
 \end{aligned} \right. \right\} .
 \label{cha}  
\end{align} 
  
 \begin{theorem}\label{thm_main3}
  
 $\A$   is the $\D$-attractor of the Volterra semiflow $\T$ in $\C$,  with properties
 \begin{itemize} 
 \item[(i)] $\A$ is precompact;
 \item[(ii)]  $\A$ is invariant under $\T$, i.e., $\T_t \A =\A$ for all $t\geq 0$;
 \item[(iii)]  $\A$ is $\D$-attracting, i.e., for any $ \mathfrak D\in \D$, 
$
  \lim_{t\to \infty}  \dist_\C (\T_t \mathfrak D, \, \A  ) =0 ;
 $
 \item[(iv)] $\A$ is the minimal  set in $\C $ which is $\D$-attracting,  that is, if $\A'\subset \C$ is also $\D$-attracting then $\A\subset \A'$. 
 \end{itemize}
 \end{theorem} 
 \bproof  
 Since we have constructed a compact $\D$-absorbing set $\B$, $\mathfrak A$ such defined is a  precompact subset of $\B$.  The other properties are proved as the standard global attractor theory, see, e.g., Robinson \cite{rob}, and also \cite{BCL,KY} for nonautonomous attractors. 
 \eproof
 
 \begin{remark}
 Due to the fact that the absorbing set $\B$ need not belong to the attraction universe $\D$ and thus need not be attracted by $\A$, we  have only the precompactness of $\A$, while  the compactness   of $\A$ remains open. 
 \end{remark} 
  
 \begin{theorem}\label{thm_main4}
    The minimal  attracting set $\Omega^*$ of  \ceqref{cfde}, defined by \ceqref{def_O}, is characterized by
  \[
   \Omega^* =\big \{ f(0) : f\in \A  \big \} .
 \]
     \end{theorem}  
    \bproof 
 By \ceqref{cha}, for any $f\in \A$ there exists a bounded sequence $\{x_{n}\}_{n\in \N}  $ in $ \R^d$  and $t_n\to \infty$
 such that $d_\C(\T_{t_n} id_{x_{n}} ,\,  f )\to0$ as $n\to \infty$.   This convergence implies 
 \[
 x(t_n, x_n) = ( \T_{t_n} id_{x_{n}} )(0) \to f(0)  ,\quad \text{as $n\to \infty$.}
 \]
Hence, $f(0) \in \Omega^*$.  Therefore, $\big \{ f(0) : f\in \A  \big \}  \subset \Omega^*$. 

On the other hand, for any $ y\in \Omega^*$, there would be a bounded sequence $\{ x_n\} $ in $\R^d$ and $t_n\to \infty$ such that $x(t_n, x_n)\to y$ as $n\to \infty$.  Since $id_{x_n} $ was attracted by $\mathfrak A$, i.e.,  up to a subsequence, $\T_{t_n} id_{x_n} $  converges to some $f\in \A$.  By the previous analysis we obtain  $y= f(0)$. 
 Hence, $y\in \{f(0):f\in \A\} $ and then $\Omega^* \subset \big \{ f(0) : f\in \A  \big \} $.  
  \eproof
     
   \subsection{Regularity of the attractor} 

\begin{theorem} \label{thm_r} 
The functions consisting the $\D$-attractor  $\A$ are equi globally H\"older continuous: 
\begin{equation} \label{eq6.13.1}
\|f(\theta_1) -f(\theta_2)\|  \leq  
 \frac {G_* }{\alpha \Gamma(\alpha )} (\theta_1  - \theta_2 )^\alpha , \quad \forall f\in \mathfrak A, \, \theta_1, \theta_2\geq 0,
\end{equation} 
where $G_*$ is the constant given by \ceqref{Gs}.  
\end{theorem} 
\bproof  
For any $f\in \A$ there exists a bounded sequence $x_n $ in $\R^d$ and $t_n \to \infty$  such that 
$d_{\C} ( \T_{t_n} id_{x_n}, f) \to 0$ as $n\to \infty$. Assume that $\sup_{n\in \N} \| x_n\| \leq R$ for some $R>0$.

Let $\theta_1 >\theta_2 \geq 0$  and $h_n:= \T_{t_n} id_{x_n}$.  The   convergence   $d_{\C} ( h_n , f) \to 0$ implies 
\[
 \sup_{\theta\in [\theta_2 , \theta_1]} \| h_n(\theta) -f(\theta)\| \to 0,\quad \text{as } n\to \infty. 
\]
In addition, by \ceqref{eq6.7.11}, 
\be 
 \big \| h_n(\theta_1)  -  h_n(\theta_2) \big \|  &
 = \big \|  (\T_{t_n} id_{x_n} )(\theta_1) - (\T_{t_n} id_{x_n} )(\theta_2) \big\| 
 \\[0.8ex]
 &\leq \sup_{\|x_0\|\leq R}  \big \|  (\T_{t_n} id_{x_0} )(\theta_1) - (\T_{t_n} id_{x_0} )(\theta_2) \big\|  \\
&  \leq
 \frac {G_* }{\alpha \Gamma(\alpha )} (\theta_1  - \theta_2 )^\alpha     + \frac{ 2G_R t_n^{\alpha-1} }{\Gamma(\alpha)}     (\theta_1 -\theta_2)  \notag
\ee
for $n$ large enough, where $G_R$ is the constant given by \ceqref{def_G}.  Hence, for $n$ large,
\ben
\|f(\theta_1) -f(\theta_2)\|  &\leq 
\|f(\theta_1) - h_n(\theta_1) \| + \| h_n(\theta_1) -h_n(\theta_2) \| + \| h_n(\theta_2) -f(\theta_2) \| \\
&\leq   \sup_{\theta\in [\theta_2, \theta_1]} \! 2\| f(\theta) -h_n(\theta)\|  + \frac {G_* }{\alpha \Gamma(\alpha )} (\theta_1  - \theta_2 )^\alpha   +  \frac{2G_R t_n^{\alpha-1} }{\Gamma(\alpha)}     (\theta_1 -\theta_2)  .
\ee
Passaging to the limit as $n\to \infty$ yields \ceqref{eq6.13.1}.  
\eproof

\begin{remark}
Compared to   the `local' H\"older continuity \ceqref{ab_f} enjoyed by the (unclosed) absorbing set $\B_0$, the H\"older continuity \ceqref{eq6.13.1}  enjoyed by the attractor $\A$  here is global with a refined coefficient,  so we have the inclusion 
 $$
 \A \subset \B_0, $$   
  which is rather nontrivial   since the  absorbing set  $\B_0$ is not closed. 
 \end{remark}


\end{document}